\documentclass[11pt,leqno]{amsart}

\usepackage{amsmath,amssymb,amsfonts,amstext,amsthm}
\usepackage{mathtools,latexsym,amscd}
\usepackage[T1]{fontenc}
\usepackage[utf8]{inputenc}
\usepackage{microtype}
\usepackage{color}
\usepackage{parskip}
\usepackage{enumitem}
\usepackage[pagebackref,colorlinks]{hyperref}

\allowdisplaybreaks[1]
\setlist[enumerate]{leftmargin=2em,itemsep=3pt,topsep=5pt}

\theoremstyle{plain}
\newtheorem{theorem}{Theorem}
\newtheorem{lemma}[theorem]{Lemma}

\newtheorem{proposition}[theorem]{Proposition}

\theoremstyle{definition}

\newtheorem*{remark*}{Remark}

\newcommand{\R}{\mathbb R}
\newcommand{\Z}{\mathbb Z}
\newcommand{\F}{F_2}
\newcommand{\Hb}{H_b}
\newcommand{\rHb}{\overline H_b}
\newcommand{\rEH}{\overline{EH}_b}
\newcommand{\Sym}{\operatorname{Sym}}
\newcommand{\Homeo}{\operatorname{Homeo}}
\newcommand{\Mod}{\operatorname{Mod}}
\newcommand{\id}{\operatorname{id}}
\newcommand{\im}{\operatorname{im}}
\newcommand{\eps}{\varepsilon}
\newcommand{\cG}{\mathcal T_M}
\newcommand{\cN}{\mathcal T_{N_g}}
\newcommand{\norm}[1]{\left\lVert#1\right\rVert}
\newcommand{\pair}[2]{\left\langle#1,#2\right\rangle}
\newcommand{\dd}{\,\mathrm d}

\hypersetup{
  pdftitle={Exact bounded cohomology of measure-preserving homeomorphism groups of nonorientable hyperbolic manifolds},
  pdfauthor={Michael Brandenbursky},
  pdfsubject={Reduced exact bounded cohomology of measure-preserving homeomorphism groups},
  pdfkeywords={bounded cohomology, non-orientable manifolds, homeomorphism groups, transfer maps}
}

\begin{document}
\raggedbottom
\title[Exact bounded cohomology of measure-preserving homeomorphism groups]
{Exact bounded cohomology of measure-preserving homeomorphism groups of nonorientable hyperbolic manifolds}
\author{Michael Brandenbursky}\thanks{M.B. was partially supported by the Israel Science Foundation grant 823/23.}
\address{Department of Mathematics, Ben Gurion University, Israel}
\email{brandens@bgu.ac.il}
\date{}

\begin{abstract}
Let $M$ be a closed, connected, non-orientable hyperbolic
manifold, and let $\mu$ be a good measure on $M$. Denote by
$\Homeo_0(M,\mu)$ the identity component of the group of
$\mu$-preserving homeomorphisms. For every $d\geq2$, we prove that
$$
\dim\rEH^d\bigl(\Homeo_0(M,\mu)\bigr)=\infty,
$$
where $\rEH^d$ denotes reduced exact bounded cohomology
in degree $d$.
\end{abstract}


\maketitle

\markboth{Exact bounded cohomology of measure-preserving homeomorphism groups}
{Exact bounded cohomology of measure-preserving homeomorphism groups}

\section{Introduction}\label{sec:introduction}

Gromov's seminal work \cite{Gromov} established bounded cohomology as a
fundamental tool in geometry and topology, with applications to the geometry
of manifolds, stable commutator length, amenable groups, and symplectic
geometry. In low degrees it is understood for many classes of finitely
generated groups. In higher degrees, however, even the bounded cohomology of
the non-abelian free group $F_2$ remains much less understood; see
\cite{Frigerio}.

Since the early 2000s, bounded cohomology of transformation groups has
received considerable attention. Degree-two exact bounded cohomology is described by
homogeneous quasimorphisms modulo homomorphisms.
Quasimorphisms have applications to dynamics and symplectic
geometry; see \cite{Entropy,G-invariant,EntovPolterovich2003}. 
Higher-degree bounded classes were studied by Brandenbursky-Marcinkowski
\cite{BrandenburskyMarcinkowski2022}, Kimura \cite{Kimura2025}, Nitsche \cite{Nitsche}
and Brandenbursky-Menashe \cite{BrandenburskyMenashe2026} 
using transfer maps from fundamental groups and braid groups.
In \cite{Brandenbursky2026}, these maps were combined with external
products to construct bounded cohomology classes in every degree
greater than one. In this paper we apply this construction to
closed non-orientable hyperbolic manifolds.

Let $M$ be a closed, connected, non-orientable hyperbolic manifold.
A finite Borel measure $\mu$ on $M$ is called \emph{good} if it is
nonatomic and positive on every nonempty open set. Let
$\cG=\Homeo_0(M,\mu)$
denote the identity component of the group of $\mu$-preserving
homeomorphisms, equipped with the compact-open topology. 
Throughout the paper, cohomology is taken for groups
regarded as discrete. Denote by
$EH_b^d(\cG)$ exact bounded cohomology and by $\rEH^d(\cG)$
reduced exact bounded cohomology, with trivial real coefficients;
the definitions are recalled in Section~\ref{sec:preliminaries}.
Our main result is the following.

\begin{theorem}\label{thm:main}
Let $M$ be a closed, connected, non-orientable hyperbolic manifold,
and let $\mu$ be a good measure. Then, for every $d\ge2$,
$$
 \dim\rEH^d(\Homeo_0(M,\mu))=\infty.
$$
In particular, $\dim EH_b^d(\Homeo_0(M,\mu))=\infty$ for every
$d\ge2$.
\end{theorem}

We divide the proof according to the dimension of $M$.
The Oxtoby--Ulam theorem
\cite[Part~II, Corollary~1]{OxtobyUlam1941} allows us to
reduce both cases to a measure induced by a smooth positive
density; see Lemma~\ref{lem:good-measures}.

Let $N_g$ denote the closed non-orientable surface of genus $g$,
and $\Sigma_g$ the closed oriented surface of genus $g$.
In dimension two, $M=N_g$ for some $g>2$, and its orientation
double cover is $\Sigma_{g-1}$. Using this cover, we prove in
Lemma~\ref{lem:infinite-disc-restriction} that, for $n\ge3$,
the degree-two exact bounded cohomology of the mapping class
group of $N_g$ with $n$ marked points has infinite-dimensional
image under restriction to the braid subgroup supported
in a disc containing those points.

For the surface case of Theorem~\ref{thm:main}, we apply this
lemma with $n=3$ and choose one class with nonzero disc
restriction. Pullback along the point-pushing homomorphism
gives an exact degree-two bounded cohomology class on the
surface braid group $B_3(N_g)$ whose restriction to the disc
braid group remains nonzero. Lemma~\ref{lem:braid-class}
shows that this class also remains nonzero on the free
subgroup generated by the squares of the two standard
Artin generators of the disc braid group. This restriction
is detected by the local disc rotations used in the proof.

We then combine the one-point maps of
\cite{BrandenburskyMenashe2026} with the mixed transfer of
\cite[Section~5.1]{Brandenbursky2026}. The transfer of the
braid class supplies a fixed degree-two factor. 
We take cup products of its powers with the degree-two and
degree-three one-point classes of Brandenbursky-Menashe~\cite{BrandenburskyMenashe2026}. 
The one-point factor ensures exactness, while the testing homomorphisms establish
injectivity as the one-point class varies. This gives the
infinite-dimensional families of reduced exact bounded
cohomology classes required in the surface case.

When $\dim M\ge3$, we choose a hyperbolically embedded free
subgroup of rank two in $\pi_1(M)$ that lies in the kernel
of the orientation character. We realize this subgroup by an embedded figure-eight.
Since the subgroup lies in the kernel of the orientation
character, a regular neighborhood of this figure-eight
is orientable. After choosing an orientation on this
neighborhood, the restriction of the density is represented
by a positive volume form. Moser's theorem allows us to
carry out the local volume-preserving pushing construction
for this form. The resulting local maps allow us to apply
the transfer argument of
\cite[Theorem~1, case~(iii)]{Brandenbursky2026}, completing
the proof in higher dimensions.

\textbf{Acknowledgments.}
The author acknowledges the support of Israel Science Foundation grant
823/23. He thanks the Department of Mathematics at Ben Gurion University for
excellent working conditions. The author used an LLM for computational exploration, algebraic checks,
and editorial assistance. The mathematical ideas, research direction, and the transfer method
approach were developed exclusively by the author.

\section{Preliminaries}\label{sec:preliminaries}

\emph{Bounded cohomology.}
Throughout the paper, groups are discrete and all
coefficients are trivial and real. Let $K$ be a discrete group.
We use right-homogeneous cochains, as in
\cite{Brandenbursky2026}. Thus $C_b^d(K)^K$ is the space of
bounded functions $c:K^{d+1}\to\R$ satisfying
$c(g_0h,\ldots,g_dh)=c(g_0,\ldots,g_d)$, $(g_0,\ldots,g_d,h\in K)$.
The differential is given by
$$
(\delta c)(g_0,\ldots,g_{d+1})
=\sum_{j=0}^{d+1}(-1)^jc(g_0,\ldots,\widehat g_j,\ldots,g_{d+1}).
$$
The cohomology of this complex is $\Hb^d(K)$. The supremum norm on cochains
induces the canonical seminorm on $\Hb^d(K)$. We write
$$
 \rHb^d(K)=\Hb^d(K)\big/\{u\in\Hb^d(K):\norm{u}=0\}
$$
for reduced bounded cohomology, with the induced quotient norm.
For a bounded cocycle $c$, the notation $[c]$ denotes its class
in the bounded or reduced bounded cohomology group indicated.
A function $\varphi:K\to\R$ is a \emph{homogeneous quasimorphism}
if $\varphi(g^j)=j\varphi(g)$ for all $g\in K$ and $j\in\Z$, and
$$
 \sup_{g,h\in K}|\varphi(gh)-\varphi(g)-\varphi(h)|<\infty.
$$
For such a function, $\delta\varphi$ denotes, in our convention,
the bounded cocycle
$$
 (\delta\varphi)(g_0,g_1,g_2)
 =\varphi(g_0g_1^{-1})+\varphi(g_1g_2^{-1})
                         -\varphi(g_0g_2^{-1}).
$$

Dropping the boundedness condition gives the ordinary cochain
complex $C^\bullet(K)^K$, whose degree-$d$ cohomology is
$H^d(K;\R)$.
Let $c_K^d$ denote the comparison map induced by inclusion of
bounded cochains into this complex. Following
\cite{FrigerioPozzettiSisto2015}, let $EH_b^d(K)$ denote its
kernel, called the exact bounded cohomology of $K$. We write
$\rEH^d(K)$ for the reduced exact bounded cohomology:
$$
 \begin{aligned}
 EH_b^d(K)
   &=\ker\bigl(c_K^d:\Hb^d(K)\longrightarrow H^d(K;\R)\bigr),\\
 \rEH^d(K)
   &=\im\bigl(EH_b^d(K)\longrightarrow\rHb^d(K)\bigr).
 \end{aligned}
$$
Thus $\rEH^d(K)$ is $EH_b^d(K)$ modulo its zero-seminorm
subspace, with the norm induced from $\rHb^d(K)$.
Every class in $EH_b^2(K)$ is represented by
$\delta\varphi$ for a homogeneous quasimorphism $\varphi$.
For such a quasimorphism, $[\delta\varphi]=0$ in $\Hb^2(K)$
if and only if $\varphi$ is a homomorphism.
See \cite{Calegari, Frigerio2017} for the basic properties.

Pullback by a group homomorphism is seminorm nonincreasing and
commutes with comparison. A group isomorphism therefore induces
isometric isomorphisms on reduced bounded cohomology and on
reduced exact bounded cohomology.

Let $\F=\langle a,b\rangle$ be the free group of rank two.
We shall use
\begin{equation}\label{eq:free-dimensions}
 \dim\rHb^2(\F)=\dim\rHb^3(\F)=\infty;
\end{equation}
see \cite{Brooks1981,Soma1997,Farre2018}.
Since $H^q(\F;\R)=0$ for $q\ge2$, these are also reduced exact
bounded cohomology groups.

\emph{Good measures and densities.}
A smooth positive density on a smooth $n$-manifold $M$ is a
positive smooth section of $|\Lambda^nT^*M|$. On a closed
manifold it defines a good measure.
On an oriented open set $U$, it is represented by a unique
positive volume form $\omega_U$. In particular, every
Riemannian metric determines a smooth positive density.

\begin{lemma}\label{lem:good-measures}
For each $M$, it suffices to prove Theorem~\ref{thm:main} for
one fixed probability measure $\mu$ induced by a smooth
positive density.
\end{lemma}
\begin{proof}
Rescaling a good measure $\mu$ does not change
$\Homeo_0(M,\mu)$, so assume that $\mu(M)=1$.
Fix a probability measure $\nu$ induced by a smooth positive
density. By the Oxtoby--Ulam theorem
\cite[Part~II, Corollary~1]{OxtobyUlam1941}, there is a
homeomorphism $h$ such that $h_*\mu=\nu$; see also
\cite[Theorem~3.1]{Fathi1980}. Conjugation gives a topological
group isomorphism
$$
 h\,\Homeo_0(M,\mu)\,h^{-1}=\Homeo_0(M,\nu).
$$
The pullbacks induced by this conjugation isomorphism and its
inverse are inverse isometries of bounded cochain complexes.
They commute with the ordinary
comparison maps. Hence they induce seminorm-preserving
isomorphisms on exact bounded cohomology and isometric
isomorphisms on reduced exact bounded cohomology.
\end{proof}

From now on we fix a hyperbolic metric on $M$, and let $\mu$ be
the measure induced by its normalized Riemannian density.
Let $p:\widetilde M\to M$ be the orientation double cover,
with its canonical orientation, and let $\tau$ be its
nontrivial deck transformation. Write $\widetilde\omega$
for the volume form of the lifted metric, rescaled so that
$\int_{\widetilde M}\widetilde\omega=2$. Then
$\tau^*\widetilde\omega=-\widetilde\omega$, and explicitly
$$
 \mu(A)=\frac12\int_{p^{-1}(A)}\widetilde\omega
 \qquad(A\subset M\text{ Borel}).
$$
Thus $\mu(M)=1$. The canonical orientation-preserving lift
of every $\mu$-preserving homeomorphism preserves the measure
induced by $\widetilde\omega$. For $M=N_g$, this is the
measure used in \cite[Section~1]{BrandenburskyMenashe2026}.
All subsequent constructions and both cases of the proof of
Theorem~\ref{thm:main} use this fixed measure.
Lemma~\ref{lem:good-measures} then gives the result for every
good measure.

\emph{External products.}
Let $K_1,\ldots,K_k$ be groups, where $k\ge1$, and write
$p_j:K_1\times\cdots\times K_k\to K_j$ for the coordinate
projections. Denote the cup product by $\smile$.
For $u_j\in\Hb^{d_j}(K_j)$, define
$$
u_1\times\cdots\times u_k=p_1^*u_1\smile\cdots\smile p_k^*u_k
\in\Hb^{d_1+\cdots+d_k}(K_1\times\cdots\times K_k).
$$
Cup and external products descend to reduced bounded cohomology
and satisfy
$$
\norm{u\smile v}\le\norm{u}\norm{v},\qquad\norm{u\times v}\le\norm{u}\norm{v}.
$$
When all factors are a group $K$, the symmetric group $S_k$
acts on $K^k$ by
$$
\sigma(g_1,\ldots,g_k)=(g_{\sigma^{-1}(1)},\ldots,g_{\sigma^{-1}(k)}).
$$
For $u\in\Hb^d(K^k)$, put
$$
\Sym_k(u)=\sum_{\sigma\in S_k}\sigma^*u.
$$
We use this unnormalized sum also on reduced bounded cohomology.

Write $Z_2^{\ell^1}(\F)$ for the degree-two cycles in the
$\ell^1$-completed bar complex, and use $\pair{\cdot}{\cdot}$
for evaluation of bounded cohomology on these cycles.
Fix a nonzero class $\theta\in\Hb^2(\F)$ and a cycle
$z\in Z_2^{\ell^1}(\F)$ such that
$\pair{\theta}{z}=1$.
Such a cycle exists by degree-two
bounded-cohomology--$\ell^1$-homology duality; see
\cite[Lemmas~7--8]{Brandenbursky2026} and
\cite{MatsumotoMorita1985}.

For $m\ge0$ and $q\in\{2,3\}$, put $d=2m+q$ and $k=m+1$.
The notation $\theta^{\times m}$ means the external product
of $m$ copies of $\theta$; when $m=0$, this is the
degree-zero unit. Define
\begin{equation}\label{eq:ordered-product}
 E_m^q:\Hb^q(\F)\longrightarrow\Hb^d(\F^k),
 \qquad E_m^q(v)=\theta^{\times m}\times v,
\end{equation}
and
\begin{equation}\label{eq:symmetric-product}
 T_m^q:\Hb^q(\F)\longrightarrow\Hb^d(\F^k),\qquad
 T_m^q(v)=\Sym_k(\theta^{\times m}\times v).
\end{equation}
Both maps descend to reduced bounded cohomology.
We use the same notation for these induced maps.

\begin{proposition}[{\cite[Lemma~8 and Proposition~9]{Brandenbursky2026}}]
\label{prop:products}
For $q=2,3$ and every $m\ge0$, the maps $E_m^q$ and $T_m^q$
are injective. Their reduced versions have bounded linear
left inverses.
\end{proposition}
\begin{proof}
We recall the calculation of the left inverses from
\cite[Lemma~8 and proof of Proposition~9]{Brandenbursky2026}.
When $m=0$, both maps are the identity. For $m\ge1$, let
$Z_m=z^{\times m}\in Z_{2m}^{\ell^1}(\F^m)$ be the chain
cross product of $m$ copies of $z$. The bounded slant-product map
$$
L_{Z_m}^*:\Hb^{2m+q}(\F^m\times\F)\longrightarrow\Hb^q(\F)
$$
evaluates the first $m$ factors against $Z_m$; see
\cite[Lemma~8]{Brandenbursky2026}. It satisfies
$$
L_{Z_m}^*E_m^q(v)=v.
$$
For the symmetric product, the calculation in
\cite[Proposition~9]{Brandenbursky2026} gives
$$
L_{Z_m}^*T_m^3(v)=m!v,\qquad
L_{Z_m}^*T_m^2(v)=m!\bigl(v+m\pair{v}{z}\theta\bigr).
$$
The operator $P:\Hb^2(\F)\to\Hb^2(\F)$ given by
$P(v)=\pair{v}{z}\theta$ satisfies $P^2=P$, since
$\pair{\theta}{z}=1$.
Thus
$$
(\id+mP)^{-1}=\id-\frac{m}{m+1}P.
$$
Thus $L_{Z_m}^*$ is a left inverse of $E_m^q$,
$(1/m!)L_{Z_m}^*$ is a left inverse of $T_m^3$, and
$(1/m!)(\id-\frac{m}{m+1}P)L_{Z_m}^*$ is a left inverse of
$T_m^2$. These identities hold before reduction, and all
three left inverses descend to bounded maps on reduced bounded
cohomology.
\end{proof}

\emph{The product transfer.}
Fix $x_0\in M$ and put
$$
\pi=\pi_1(M,x_0),\qquad G=\pi_M=\pi/Z(\pi).
$$
Here $Z(\pi)$ is the center of $\pi$. For the manifolds
considered here $Z(\pi)=1$, so $G=\pi$.
Here $Z(\pi)$ is the center of $\pi$. For the manifolds
considered here $Z(\pi)=1$, so $G=\pi$.
Fix an auxiliary Riemannian metric on $M$. Outside the cut
locus of $x_0$, let $\alpha_x$ be the unique minimizing
geodesic from $x_0$ to $x$. The cut locus has $\mu$-measure
zero, since $\mu$ is induced by a smooth positive density.
Choose the remaining paths measurably. 

For $f\in\cG$, choose an isotopy $f_t$ from the identity
to $f$. The trajectory cocycle is
$$
\gamma:\cG\times M\longrightarrow G,\qquad
\gamma(f,x)=[\alpha_x*\{f_t(x)\}*\alpha_{f(x)}^{-1}]^{-1}.
$$
The bracket denotes the based homotopy class in $G$,
and $*$ denotes path concatenation from left to right, with
multiplication $[\alpha][\beta]=[\alpha*\beta]$.
This is the construction of
\cite[Section~3]{BrandenburskyMenashe2026}, with the inverse-loop
convention of \cite[equation~(22)]{Brandenbursky2026}.
With multiplication $fg=f\circ g$, it satisfies
$$
 \gamma(fg,x)=\gamma(f,gx)\gamma(g,x).
$$
The evaluation of a loop of homeomorphisms lies in $Z(\pi)$.
Thus the cocycle is independent of the chosen isotopies.

For $k\ge1$, equip $M^k$ with the product measure $\mu^k$ and
the diagonal action on $\mathbf x=(x_1,\ldots,x_k)$ given by
$h\mathbf x=(h(x_1),\ldots,h(x_k))$. Define
$$
\gamma_k:\cG\times M^k\longrightarrow G^k,\qquad
\gamma_k(f,\mathbf x)=(\gamma(f,x_1),\ldots,\gamma(f,x_k)).
$$
Define the cochain map
$\Gamma'_{b,k}:C_b^d(G^k)^{G^k}\to C_b^d(\cG)^{\cG}$ by
\begin{equation}\label{eq:product-transfer}
\begin{split}
(\Gamma'_{b,k}c)(f_0,\ldots,f_d)=\int_{M^k}c\bigl(\gamma_k(f_0,\mathbf x),\ldots,
\gamma_k(f_d,\mathbf x)\bigr)\dd\mu^k(\mathbf x).
\end{split}
\end{equation}
Right homogeneity follows from the cocycle identity and the
change of variables $\mathbf x\mapsto h\mathbf x$.
Integration commutes with $\delta$, and
$\norm{\Gamma'_{b,k}c}_\infty\le\norm{c}_\infty$.
We denote the induced maps by
$$
\Gamma_{b,k}^d:\Hb^d(G^k)\longrightarrow\Hb^d(\cG),\qquad
\overline\Gamma_{b,k}^{\,d}:\rHb^d(G^k)\longrightarrow\rHb^d(\cG).
$$

By \cite[Lemma~2.2]{BrandenburskyMarcinkowski2022}, for each
$f\in\cG$ there is a finite set $S_f\subset G$ such that
$$
\gamma(f,x)\in S_f\qquad\text{for almost every }x\in M.
$$
The proof applies without an orientability assumption.
We use this finite-image property as in
\cite[proof of Lemma~10, equation~(24)]{Brandenbursky2026}.
For each fixed tuple $f_0,\ldots,f_d\in\cG$, the map
$$
\mathbf x\longmapsto\bigl(\gamma_k(f_0,\mathbf x),\ldots,\gamma_k(f_d,\mathbf x)\bigr)
$$
takes values almost everywhere in the finite set
$S_{f_0}^{\,k}\times\cdots\times S_{f_d}^{\,k}$.
Hence, even for an ordinary cochain $c$ which is not bounded,
the integrand in \eqref{eq:product-transfer} takes only
finitely many values outside a null set. Since $\mu$ is
finite, this integrand is integrable. Thus
\eqref{eq:product-transfer} also defines a map on ordinary
cochains. Denote the resulting maps by
$$
 \Gamma'_k:C^d(G^k)^{G^k}\longrightarrow C^d(\cG)^{\cG},\qquad
 \Gamma_k^d:H^d(G^k;\R)\longrightarrow H^d(\cG;\R).
$$
The comparison identity
$$
 c_{\cG}^d\Gamma_{b,k}^d
 =\Gamma_k^d c_{G^k}^d
$$
shows that the transfer preserves exact classes; see
\cite[Lemma~2.2]{BrandenburskyMarcinkowski2022} and
\cite[Lemma~10]{Brandenbursky2026}.
These constructions do not require an orientation of $M$.

\emph{Configuration spaces.}
For $g>2$ and $r\ge1$, let
$$
 X_r(N_g)=\{(x_1,\ldots,x_r)\in N_g^r:
                   x_i\ne x_j\text{ for }i\ne j\}.
$$
The symmetric group $S_r$ acts by permuting the coordinates.
With chosen basepoints write
$$
P_r(N_g)=\pi_1(X_r(N_g)),\qquad
B_r(N_g)=\pi_1(X_r(N_g)/S_r)
$$
for the pure and full braid groups, respectively.
We use the same notation for configuration spaces and braid
groups of a disc, using its interior when the disc is closed.
We put
$$
 \nu=\mu^3|_{X_3(N_g)},\qquad
 \cN=\Homeo_0(N_g,\mu).
$$
Since $\mu$ is nonatomic and $\mu(N_g)=1$, the measure $\nu$
also has total mass one.
Choose measurable closing paths in $X_3(N_g)$ from a fixed
base configuration. Applying an isotopy to all three points
and closing the resulting path, with the same inverse-loop
convention as for $\gamma$, defines the trajectory cocycle
$$
\gamma_{N_g,3}:\cN\times X_3(N_g)\longrightarrow P_3(N_g)\subset B_3(N_g).
$$
The pure braid groups $P_r(N_g)$ are center-free for every
$r\ge1$; see \cite[proof of Proposition~1.6]{ParisRolfsen1999}.
Thus $\gamma_{N_g,3}$ is independent of the chosen isotopy.

We define the cochain map
$$
\Gamma'_{N_g,3}:C_b^d(B_3(N_g))^{B_3(N_g)}\longrightarrow C_b^d(\cN)^{\cN}
$$
by \eqref{eq:product-transfer}, replacing
$M^k,\mu^k,\gamma_k$ by
$X_3(N_g),\nu,\gamma_{N_g,3}$, respectively;
see \cite{Kimura2025}.
Following \cite[Sections~4-5]{Brandenbursky2026}, we use the notation
$$
\Gamma_{N_g,3}^d:\rHb^d(B_3(N_g))\longrightarrow\rHb^d(\cN)
$$
for the induced reduced transfer.
For $m\ge0$, the product space $X_3(N_g)^m\times N_g$, with
measure $\nu^m\times\mu$, carries the diagonal action of $\cN$.
The product of the three-point cocycles on the first $m$
factors and the one-point cocycle on the last factor takes
values in $B_3(N_g)^m\times\pi_1(N_g)$. Integrating bounded
cochains against this product cocycle defines the mixed transfer used in
Section~\ref{sec:surfaces}. All cochains integrated in these
configuration-space constructions are bounded.

\section{Proof of Theorem 1 in dimension two}\label{sec:surfaces}

Let $g>2$, $G=\pi_1(N_g)$ and  
$\overline\Gamma_{b,1}^{\,q}$ denotes the one-point transfer on $N_g$.

\emph{The one-point classes.}
We briefly recall the construction of Brandenbursky--Menashe
\cite{BrandenburskyMenashe2026} and the proof of their main
theorem, which give the degree-two and degree-three classes
used below.
Choose an oriented one-holed torus $Q\cong\Sigma_{1,1}\subset N_g$
such that $N_g\setminus\operatorname{int}Q\cong N_{g-2,1}$.
Here $N_{g-2,1}$ is the compact non-orientable surface of genus
$g-2$ with one boundary component. Choose a common basepoint
for the meridian and longitude of $Q$ and a path from $x_0$
to it. Let
$i:\F\hookrightarrow\pi_1(N_g)$ send $a,b$ to their based
homotopy classes. Then
$$
\pi_1(N_g)=\F*_{\Z}\pi_1(N_{g-2,1}).
$$
For $q=2,3$, the isometric section theorem of
\cite{BucherEtAl2014} gives isometric linear sections of restriction
$$
J_q:\rHb^q(\F)\longrightarrow\rHb^q(\pi_1(N_g)),\qquad i^*J_q=\id;
$$
see \cite[Proposition~4.10]{BrandenburskyMenashe2026}.
Since $N_g$ is closed, aspherical and non-orientable,
\begin{equation}\label{eq:surface-ordinary}
 H^q(\pi_1(N_g);\R)=0\qquad(q\ge2).
\end{equation}
Thus $J_q(v)\in\rEH^q(\pi_1(N_g))$ for every
$v\in\rHb^q(\F)$. Put
$$
\Theta_q:\rHb^q(\F)\longrightarrow\rEH^q(\cN),\qquad
\Theta_q(v)=\overline\Gamma_{b,1}^{\,q}(J_q(v)).
$$

For sufficiently small $\delta>0$, let
$s_\delta:\F\to\cN$ be the annular pushing homomorphisms of
\cite[Lemma~4.12]{BrandenburskyMenashe2026}.
Their defining isotopies are supported in annuli about the
chosen meridian and longitude of $Q$, with transition collars
whose measure tends to zero as $\delta\to0$.
Choose the pushing directions so that, for almost every
$x$ in the overlap of the two annuli outside their
transition collars, there is an element
$u_x\in\pi_1(N_g)$ such that
$$
\gamma(s_\delta(w),x)=u_x\,i(w)\,u_x^{-1}\qquad\text{for every }w\in\F.
$$
On $Q$, the density defining $\mu$ is a positive area form,
so the boundary version of Moser's theorem
\cite{Moser1965,Banyaga1974} gives this local construction.
There is a constant $\kappa>0$ such that, for $q=2,3$ and
every fixed $v\in\rHb^q(\F)$,
\begin{equation}\label{eq:one-point-limit}
 \norm{s_\delta^*\Theta_q(v)-\kappa v}\longrightarrow0
 \qquad(\delta\to0).
\end{equation}
The maps and the constant are independent of $q$ and $v$.
If $\Theta_q(v)=0$, then \eqref{eq:one-point-limit} gives
$\norm{\kappa v}=0$. Since we work in reduced bounded cohomology
and $\kappa>0$, this implies $v=0$, and thus $\Theta_q$ is injective.

\emph{Restriction to disc braid groups.}
Let $\Mod(N_g,n)$ denote the full mapping class group of $N_g$
with $n$ marked points.
Fix an embedded oriented disc $D\subset N_g$ whose interior
contains these points. Extension by the identity gives an injective homomorphism
$$
 \iota_D:B_n(D)\hookrightarrow\Mod(N_g,n).
$$
Indeed, this map is the composition of the homomorphism
$B_n(D)\to B_n(N_g)$ induced by the inclusion
$D\hookrightarrow N_g$, which is injective by
\cite[Theorem~2.3]{ParisRolfsen1999}, and the point-pushing
homomorphism $B_n(N_g)\to\Mod(N_g,n)$, which is injective
for $g>2$ by
\cite[Theorem~1.1, diagram~(1.4)]{GoncalvesGuaschiMaldonado2018}.

\begin{lemma}\label{lem:infinite-disc-restriction}
Let $g>2$ and $n\ge3$. Then the image of the restriction map
$$
\iota_D^*:EH_b^2(\Mod(N_g,n))\longrightarrow EH_b^2(B_n(D))
$$
is infinite dimensional. The same assertion holds for reduced
exact bounded cohomology.
\end{lemma}
\begin{proof}
Choose a smaller disc $E\subset D$ containing exactly three
marked points. Regard $B_3(E)$ as a subgroup of $B_n(D)$ by
fixing the remaining strands. Let $\sigma_1,\sigma_2$ be its
standard Artin generators, and let $\Delta_3^2$ denote the
full twist. For every integer $r\ge3$, put
\begin{equation}\label{eq:braid-family}
\Delta_3^2=(\sigma_1\sigma_2)^3,\qquad u_r=\sigma_2^{-1}\sigma_1\sigma_2^{-r}\sigma_1^2,\qquad
v_r=u_r^6(\Delta_3^2)^{\,r-2}.
\end{equation}
The exponent-sum homomorphism $B_3(E)\to\Z$, which sends
each Artin generator to $1$, is the abelianization.
The exponent sums of $u_r$ and $\Delta_3^2$ are $2-r$ and
$6$, respectively. Hence the exponent sum of $v_r$ is zero, and therefore
\begin{equation}\label{eq:braid-commutators}
 v_r\in[B_3(E),B_3(E)]\subset[B_n(D),B_n(D)].
\end{equation}
Put
$$
R=\begin{pmatrix}1&1\\0&1\end{pmatrix},\qquad
L=\begin{pmatrix}1&0\\1&1\end{pmatrix},
$$
and write $[A]$ for the projective class of a matrix $A$.
Use the identification
$B_3(E)/\langle\Delta_3^2\rangle\cong\operatorname{PSL}(2,\Z)$
given by $\sigma_1\mapsto[R]$ and $\sigma_2\mapsto[L^{-1}]$.
Then $u_r$ maps to $[M_r]$, where
$$
 M_r=LRL^rR^2
   =\begin{pmatrix}r+1&2r+3\\2r+1&4r+4\end{pmatrix}.
$$
Its trace is $5(r+1)>2$, so $v_r$ is pseudo-Anosov on the
three-punctured disc $E$, with projective matrix $[M_r]^6$.
Twists about $\partial E$ are powers of $\Delta_3^2$ and
therefore have trivial image in this quotient. A conjugacy
by a homeomorphism of the marked disc, possibly reversing
orientation, induces a conjugacy of projective matrices in
$\operatorname{PGL}(2,\Z)$.

We first establish two properties of the projective matrices
$[M_r]$ in $\operatorname{PGL}(2,\Z)$: no nonzero power of
$[M_r]$ is conjugate to its inverse, and, if $r\ne s$, no
nonzero power of $[M_r]$ is conjugate to a nonzero power of
$[M_s]$.

Decompose the cyclic word into maximal blocks of consecutive
identical letters. To each block $L^a$ assign the integer $-a$,
and to each block $R^b$ assign the integer $b$, keeping the
order of the blocks. For $M_r=LRL^rR^2$, the resulting
sequence is $(-1,1,-r,2)$.
By the normal form of
\cite[Section~3.1.1, Propositions~3--4]{BaakeRoberts1997},
conjugacy to the inverse in $\operatorname{PGL}(2,\Z)$
requires this cyclic sequence to agree with its reversal or
its negative reversal. Neither is possible. The negative
block of length $r$ is preceded by a positive block of length
$1$ and followed by one of length $2$, whereas reversal
interchanges these neighbors. Negative reversal has a positive
block of length $r$, which the original sequence lacks.
The same argument applies to every repetition of the sequence.
Thus no nonzero power of $[M_r]$ is conjugate to its inverse.

The largest absolute block length is $r$. It is unchanged by
repetition, cyclic permutation, reversal, or change of all
signs. Consequently, if $r\ne s$, no nonzero powers of
$[M_r]$ and $[M_s]$ are conjugate in
$\operatorname{PGL}(2,\Z)$.

Let $p:\Sigma_{g-1}\to N_g$ be the orientation cover. Its
nontrivial deck transformation reverses orientation and will
be called the deck reflection. Write
$\Mod^+(\Sigma_{g-1},2n)$ for the group of orientation-preserving
mapping classes of the cover with the $2n$ lifted marked
points, which may be permuted. The canonical
orientation-preserving lift gives an injective homomorphism
$$
 \Lambda:\Mod(N_g,n)\hookrightarrow
                  \Mod^+(\Sigma_{g-1},2n);
$$
see \cite[Theorem~1.1]{GoncalvesGuaschiMaldonado2018}.

A suitable positive power of a mapping class admits a
Nielsen--Thurston decomposition as a commuting product of
pseudo-Anosov mapping classes supported on subsurfaces and
powers of Dehn twists; we call these factors components.
See \cite[Section~4]{BBF2016}.
For the components below, conjugacy is taken in
$\Mod^+(\Sigma_{g-1},2n)$. Following the description preceding
\cite[Lemma~4.11]{BGKM2016},
a component is \emph{chiral} if it is not conjugate to its
inverse. We will verify that every positive power of each
pseudo-Anosov component below is chiral.
By \cite[Lemma~4.11]{BGKM2016}, a mapping class is detected
by a homogeneous quasimorphism if the Nielsen--Thurston
decomposition of a suitable positive power is nonempty,
all its components are chiral, and no nonzero powers of
distinct components are conjugate.

Let $f_r=\Lambda(\iota_D(v_r))$ and
$E^+,E^-$ be the two lifts of $E$, labeled so that
$p|_{E^+}:E^+\to E$ is orientation-preserving.
The mapping class $f_r$ preserves each of these discs and
is the identity on their complement.
Its Nielsen--Thurston decomposition has two pseudo-Anosov
components, denoted by $h_r^+,h_r^-$, supported on $E^+,E^-$,
respectively, together with non-zero powers of Dehn twists about their
boundaries. Let $\tau$ denote the deck reflection.
Since $f_r$ commutes with $\tau$ and $\tau(E^+)=E^-$,
we choose the pseudo-Anosov components so that
$h_r^-=\tau h_r^+\tau^{-1}$.
Thus their restrictions to $E^+$ and $E^-$ are conjugate
by the orientation-reversing homeomorphism
$\tau|_{E^+}:E^+\to E^-$. For a pseudo-Anosov component $h$, let
$\lambda(h)>1$ denote the stretch factor of its restriction
to its supporting disc. Thus the two components have a
common stretch factor
$$
 \lambda(h_r^+)=\lambda(h_r^-)=\lambda_r>1.
$$
A conjugacy between nonzero powers of any two members of the
family $\{h_r^+,h_r^-:r\ge3\}$ carries the
boundary of one supporting disc to the boundary of the other,
up to isotopy. Indeed, these boundaries are their canonical
reduction curves, which are preserved by conjugacy; see
\cite[Section~4]{BBF2016}.
The disc side cannot be exchanged with the complementary
surface of positive genus. Thus, after an isotopy, the
conjugating homeomorphism restricts to a homeomorphism
between the supporting discs. In particular, when the
supporting disc is the same on both sides, the conjugacy
of the restrictions is induced by a homeomorphism of that disc.

Under the orientation-preserving identification $E^+\to E$,
the restriction of $f_r$ is $v_r$. Passing to
the component $h_r^+$ removes only boundary twists, which
are trivial in $B_3(E)/\langle\Delta_3^2\rangle$.
Since
$$
v_r=u_r^6(\Delta_3^2)^{r-2}\longmapsto [M_r]^6,
$$
the restriction of $h_r^+$ has projective matrix
$[M_r]^6$. The deck reflection conjugates the restriction
of $h_r^+$ to that of $h_r^-$.
If $(h_r^+)^a$ were conjugate to $(h_r^+)^{-a}$ for
some $a\ge1$, restricting the conjugacy to $E^+$ would
make $[M_r]^{6a}$ conjugate to its inverse in
$\operatorname{PGL}(2,\Z)$, contrary to the matrix argument
above. The same argument on $E^-$, identified with $E^+$
by the deck reflection, applies to $h_r^-$.
Thus every positive power of each $h_r^\pm$ is chiral.

Suppose that $(h_r^+)^a$ is conjugate to $(h_r^-)^{-b}$,
where $a,b\ge1$. Stretch factors are invariant under
conjugacy, and
$$
\lambda(h^j)=\lambda(h)^{|j|}\qquad(j\ne0).
$$
Consequently,
$$
\lambda_r^a=\lambda\bigl((h_r^+)^a\bigr)=\lambda\bigl((h_r^-)^{-b}\bigr)=\lambda_r^b.
$$
Since $\log\lambda_r>0$, we obtain $a=b$.
Let $H$ be an orientation-preserving homeomorphism
representing this conjugacy, so that
$$
H(h_r^+)^aH^{-1}=(h_r^-)^{-a}.
$$
Here and below, the equalities are equalities of mapping
classes. As explained above, after an isotopy we may assume
that $H(E^+)=E^-$.
The deck reflection $\tau$ maps $E^-$ back to $E^+$.
Consequently, $\tau H$ preserves $E^+$.
Our choice of the components gives
$\tau h_r^-\tau^{-1}=h_r^+$. Therefore
$$
(\tau H)(h_r^+)^a(\tau H)^{-1}=\tau\bigl(H(h_r^+)^aH^{-1}\bigr)\tau^{-1}
=\tau(h_r^-)^{-a}\tau^{-1}=(h_r^+)^{-a}.
$$
Thus the restriction of $\tau H$ to $E^+$ would conjugate
the restriction of $(h_r^+)^a$ to its inverse.

We now apply the projective matrix representation.
Under the identification $E^+\to E$, the restriction of
$h_r^+$ has matrix $[M_r]^6$, so its $a$-th power has
matrix $[M_r]^{6a}$.
The projective matrix representation extends to
orientation-reversing homeomorphisms of the three-punctured
disc, with values in $\operatorname{PGL}(2,\Z)$.
Since $\tau H$ reverses orientation, let
$C\in\operatorname{PGL}(2,\Z)$ denote the projective matrix
of its restriction to $E^+$.
The preceding conjugacy would then give
$$
C[M_r]^{6a}C^{-1}=[M_r]^{-6a}.
$$
This contradicts the fact proved above that no nonzero
power of $[M_r]$ is conjugate to its inverse in
$\operatorname{PGL}(2,\Z)$.
Thus no positive power of one component is conjugate
to a negative power of the other.

Fix $r\ge3$. The construction in the proof of
\cite[Theorem~4.2]{BBF2016}, together with
\cite[Proposition~4.4]{BBF2016}, gives a homogeneous
quasimorphism $\varphi$ on $\Mod^+(\Sigma_{g-1},2n)$ such that
$\varphi(h_r^+)\ne0$ and $\varphi$ vanishes on every
boundary-twist factor in the decomposition of $f_r$.
If $h_r^+$ and $h_r^-$ have no conjugate nonzero powers, we may
also choose $\varphi(h_r^-)=0$.
If nonzero powers of $h_r^+$ and $h_r^-$ are conjugate,
the preceding argument shows that their exponents have the
same sign. Equality of stretch factors then
shows that their exponents have the same absolute value.
Thus $(h_r^+)^a$ and $(h_r^-)^a$ are conjugate for some
$a\ge1$, and homogeneity and conjugacy invariance give
$\varphi(h_r^+)=\varphi(h_r^-)\ne0$.
Since homogeneous quasimorphisms are additive on commuting
elements, in both cases
$$
\varphi(f_r)=\varphi(h_r^+)+\varphi(h_r^-)\ne0.
$$

Similarly, a conjugacy between nonzero powers of pseudo-Anosov
components belonging to $r$ and $s$ restricts to a homeomorphism
between their supporting discs. Use the covering maps to identify
both discs with $E$. The resulting homeomorphism of $E$ may
reverse orientation, but it would give a conjugacy in
$\operatorname{PGL}(2,\Z)$ between nonzero powers of
$[M_r]$ and $[M_s]$. When $r\ne s$, this contradicts the
different maximal block lengths established above.

Fix $N\ge3$. Apply \cite[Proposition~4.4]{BBF2016} to the
finite collection of components of $f_3,\ldots,f_N$.
For each $i$ with $3\le i\le N$, we obtain a homogeneous
quasimorphism $\varphi_i$ on $\Mod^+(\Sigma_{g-1},2n)$
which is nonzero on $h_i^+$ and vanishes on
all boundary-twist factors and on $h_j^\pm$ for $j\ne i$.
When $h_i^+$ and $h_i^-$ have no conjugate nonzero powers,
we also choose $\varphi_i(h_i^-)=0$.
Otherwise, the preceding argument gives
$\varphi_i(h_i^-)=\varphi_i(h_i^+)$.
Additivity on commuting elements therefore gives
$\varphi_i(f_i)\ne0$ and $\varphi_i(f_j)=0$ for $j\ne i$.
After rescaling, we have
$$
\varphi_i(f_j)=\delta_{ij}\qquad(3\le i,j\le N).
$$
Put $q_i=\varphi_i\circ\Lambda$. Then
$[\delta q_i]\in EH_b^2(\Mod(N_g,n))$.
Suppose
$$
\sum_{i=3}^N a_i\iota_D^*[\delta q_i]=0.
$$
The homogeneous quasimorphism
$\sum_{i=3}^N a_i(q_i\circ\iota_D)$ represents zero in exact
bounded cohomology and hence is a homomorphism.
By \eqref{eq:braid-commutators}, it vanishes on every $v_j$.
Evaluating gives
$$
0=\sum_{i=3}^N a_iq_i(\iota_D(v_j))=a_j\qquad(3\le j\le N).
$$
The restricted classes are linearly independent.
Since $N$ is arbitrary, the image is infinite dimensional.
In degree two there are no nontrivial bounded cohomology classes
of zero seminorm; see \cite[Corollary~2.7]{MatsumotoMorita1985}.
Hence the reduced assertion follows.
\end{proof}

\emph{A class detected in a disc.}
Choose an oriented disc $D_0\subset N_g\setminus Q$ and a base
configuration of three points in its interior. In this disc, let
$\sigma_1,\sigma_2$ denote the standard Artin generators of
$B_3(D_0)$.
Let $\iota:B_3(D_0)\hookrightarrow B_3(N_g)$ be the local inclusion,
and put
$$
\jmath:\F\hookrightarrow B_3(D_0),\qquad\jmath(a)=\sigma_1^2,\quad \jmath(b)=\sigma_2^2.
$$
The image of $\jmath$ lies in $P_3(D_0)$. The subgroup
generated by $\sigma_1^2,\sigma_2^2$ is the free factor in
$P_3(D_0)\cong\F\times\Z$.

\begin{lemma}\label{lem:braid-class}
There is a homogeneous quasimorphism $q_0:B_3(N_g)\to\R$ such that
$$
\eta=\jmath^*\iota^*[\delta q_0]\ne0\quad\text{in }\rHb^2(\F).
$$
\end{lemma}
\begin{proof}
Apply Lemma~\ref{lem:infinite-disc-restriction} with $n=3$ and
$D=D_0$. It gives a homogeneous quasimorphism on $\Mod(N_g,3)$
whose restriction to $B_3(D_0)$ is not a homomorphism.
Pull it back by the point-pushing homomorphism
$B_3(N_g)\to\Mod(N_g,3)$, and denote the resulting homogeneous
quasimorphism by $q_0$. The composite
$$
B_3(D_0)\xrightarrow{\iota}B_3(N_g)\longrightarrow\Mod(N_g,3)
$$
is extension by the identity, so $\iota^*[\delta q_0]\ne0$.

Restriction to the finite-index subgroup $P_3(D_0)$ is injective
in bounded cohomology. The amenable direct factor in
$P_3(D_0)\cong\F\times\Z$ can then be removed isometrically.
Thus $\jmath^*\iota^*[\delta q_0]\ne0$.
In degree two there are no nontrivial bounded cohomology classes
of zero seminorm; see \cite[Corollary~2.7]{MatsumotoMorita1985}.
Hence this class remains nonzero after reduction.
\end{proof}

\begin{remark*}
A quasimorphism satisfying Lemma~\ref{lem:braid-class} can also
be constructed directly.
We identify $B_3(N_g)$ with its point-pushing subgroup of
$\Mod(N_g,3)$ and use the lift $\Lambda$ defined above.
Interpret the word $v_3=u_3^6\Delta_3^2$ from
\eqref{eq:braid-family} in $B_3(D_0)$, using the standard
Artin generators of this disc braid group.
The projective matrix of $u_3$ is
$[M_3]=\bigl[\begin{smallmatrix}4&9\\7&16\end{smallmatrix}\bigr]$.
No nonzero power of $[M_3]$ is conjugate to its inverse
in $\operatorname{PGL}(2,\Z)$; see
\cite[Example~2, Lemma~9 and Proposition~5]{BaakeRoberts1997}.

Applying the argument in the proof of
Lemma~\ref{lem:infinite-disc-restriction} to the single
braid $v_3$, we obtain a homogeneous quasimorphism
$\varphi$ on $\Mod^+(\Sigma_{g-1},6)$ such that
$$
\varphi\bigl(\Lambda(\iota(v_3))\bigr)\ne0.
$$
Let $q_0$ be the restriction of $\varphi\circ\Lambda$
to $B_3(N_g)$. Then $q_0(\iota(v_3))\ne0$.
Since $v_3\in[B_3(D_0),B_3(D_0)]$, the restriction of $q_0$
to $B_3(D_0)$ is not a homomorphism.
Restriction to $P_3(D_0)$ and removal of its central
$\Z$ factor, as in the proof of Lemma~\ref{lem:braid-class},
give
$\jmath^*\iota^*[\delta q_0]\ne0$ in $\rHb^2(\F)$.
\end{remark*}

\emph{The mixed product transfer.}
Let $\alpha\in\rHb^2(B_3(N_g))$ be the reduced class of
$\delta q_0$, and put
$$
A=\Gamma_{N_g,3}^2(\alpha)\in\rHb^2(\cN).
$$
Recall that $\eta=\jmath^*\iota^*\alpha\ne0\thinspace\text{in }\rHb^2(\F)$ by
Lemma~\ref{lem:braid-class}.
For $m\ge0$ and $q\in\{2,3\}$, define
\begin{equation}\label{eq:surface-map}
 \begin{aligned}
 \Phi_{m,q}&:\rHb^q(\F)\longrightarrow\rHb^{2m+q}(\cN),\\
 \Phi_{m,q}(v)&=A^{\smile m}\smile\Theta_q(v).
 \end{aligned}
\end{equation}
Here $A^{\smile0}$ is the degree-zero unit. This is the mixed
transfer of $\alpha^{\times m}\times J_q(v)$, a class on
$B_3(N_g)^m\times\pi_1(N_g)$, over $X_3(N_g)^m\times N_g$
with measure $\nu^m\times\mu$.
Indeed, the cup-product integrand factors over the independent
variables, and Fubini gives \eqref{eq:surface-map}.

For the testing argument fix $m\ge1$. Choose pairwise disjoint
discs $D_1,\ldots,D_m\Subset\operatorname{int}(D_0)$.
Choose $R>0$ sufficiently small.
Inside each $D_j$, choose two overlapping round discs
$D_{j,1},D_{j,2}$ of radius $R$ in a positively oriented
area coordinate chart, with closures in $D_j$, such that
their intersection and both exclusive regions have positive
measure. Label these three regions in the order
$D_{j,1}\setminus D_{j,2}$,
$D_{j,1}\cap D_{j,2}$,
$D_{j,2}\setminus D_{j,1}$.
Use a model configuration with one point in each region,
in the displayed order.

For sufficiently small $0<\eps<R$, choose a smooth function
$\chi_\eps:[0,R]\to[0,1]$ equal to one on $[0,R-\eps]$
and zero near $R$.
For $s=1,2$, use polar coordinates $(r,\varphi)$ about the
center of $D_{j,s}$, choosing the increasing angular direction
so that the inverse traced full rotation of the model
configuration represents $\sigma_s^2$.
As in \cite[Lemmas~15 and~21]{Brandenbursky2026}, let
$f_{j,s,\eps}$ be the time-one map of the rotation
$$
(r,\varphi)\longmapsto (r,\varphi+2\pi t\chi_\eps(r)),\qquad 0\le t\le1,
$$
on $D_{j,s}$, extended by the identity outside that disc.
The transition collars are the regions where
$R-\eps<r<R$; their total measure tends to zero
as $\eps\to0$.
Define
$$
\rho_{j,\eps}:\F\longrightarrow\cN,\qquad\rho_{j,\eps}(a)=f_{j,1,\eps},\quad
\rho_{j,\eps}(b)=f_{j,2,\eps}.
$$

Fix $j,\eps$, and let $\mathbf x\in X_3(N_g)$ have one point
in each of the three regions and no point in either transition
collar. Both generators fix $\mathbf x$. Hence,
$$
\F\longrightarrow B_3(N_g),\qquad w\longmapsto\gamma_{N_g,3}(\rho_{j,\eps}(w),\mathbf x)
$$
is a homomorphism. Order the points by their regions. Choose
an orientation-preserving identification of the pair
$D_{j,1},D_{j,2}$, together with these points, with the
standard pair of overlapping discs in $D_0$ surrounding base
points $1,2$ and $2,3$, respectively. The inverse trajectory
braids of the two rotations are then $\sigma_1^2$ and
$\sigma_2^2$. Restoring the original labels and using the
chosen closing paths conjugates both braids by the same
element. Thus there is $\beta_{\mathbf x}\in B_3(N_g)$
such that
$$
\gamma_{N_g,3}(\rho_{j,\eps}(w),\mathbf x)
=\beta_{\mathbf x}(\iota\circ\jmath)(w)\beta_{\mathbf x}^{-1}\qquad(w\in\F).
$$
For the fixed $j,\eps$, the braid $\beta_{\mathbf x}$ may
depend on $\mathbf x$ and the closing paths, but not on $w$;
see \cite[proofs of Lemmas~15 and~21]{Brandenbursky2026}.
Since inner automorphisms act trivially on bounded cohomology,
the pullback of $\alpha$ by this homomorphism is
$\jmath^*\iota^*\alpha=\eta\in\rHb^2(\F)$.

The images for different $j$ commute, and their supports are
disjoint from the support of $s_\delta$, which lies in $Q$.
Thus
\begin{equation}\label{eq:surface-testing-map}
 \begin{aligned}
 \rho_{\eps,\delta}&:\F^{m+1}\longrightarrow\cN,\\
 \rho_{\eps,\delta}(g_1,\ldots,g_m,h)
   &=\rho_{1,\eps}(g_1)\cdots\rho_{m,\eps}(g_m)s_\delta(h)
 \end{aligned}
\end{equation}
is a homomorphism; compare
\cite[equation~(63)]{Brandenbursky2026}.
Let $p_j:\F^{m+1}\to\F$ be the coordinate projections, and put
$$
\eta_{j,\eps}=\rho_{j,\eps}^*A,\qquad E_\delta=s_\delta^*A,\qquad v_\delta=s_\delta^*\Theta_q(v).
$$
Thus $\eta_{j,\eps},E_\delta\in\rHb^2(\F)$ and
$v_\delta\in\rHb^q(\F)$.
The degree-two splitting of
\cite[Lemma~22]{Brandenbursky2026} gives \eqref{eq:split}, and
\cite[Lemma~24]{Brandenbursky2026} gives \eqref{eq:local-seed}:
\begin{align}
\rho_{\eps,\delta}^*A
&=\sum_{j=1}^m p_j^*\eta_{j,\eps}+p_{m+1}^*E_\delta,\label{eq:split}\\
\rho_{\eps,\delta}^*\Theta_q(v)
&=p_{m+1}^*v_\delta.\label{eq:local-seed}
\end{align}
Let $\lambda_j=6\,\mu(D_{j,1}\setminus D_{j,2})\,\mu(D_{j,1}\cap D_{j,2})\,\mu(D_{j,2}\setminus D_{j,1})>0$.
By
\cite[Lemma~21]{Brandenbursky2026} and \eqref{eq:one-point-limit},
$$
\eta_{j,\eps}\longrightarrow\lambda_j\eta\quad(\eps\to0),\qquad
v_\delta\longrightarrow\kappa v\quad(\delta\to0)
$$
in reduced norm.
The defining isotopies of $\rho_{j,\eps}$ stay in
the contractible disc $D_j$, and hence give trivial one-point
trajectory classes. This is why the first $m$ factors do
not occur in \eqref{eq:local-seed}.

\emph{Two coordinates.}
For $m=1$, naturality of the cup product gives
$$
\rho_{\eps,\delta}^*\Phi_{1,q}(v)=\eta_{1,\eps}\times v_\delta+p_2^*(E_\delta\smile v_\delta).
$$
Let $\tau_1:\F^2\to\F^2$ be given by $\tau_1(g,h)=(1,h)$.
Applying $\id-\tau_1^*$ removes the second summand. Hence
$$
(\id-\tau_1^*)\rho_{\eps,\delta}^*\Phi_{1,q}(v)\longrightarrow\lambda_1\kappa\,\eta\times v
\qquad((\eps,\delta)\to(0,0)).
$$
If $\Phi_{1,q}(v)=0$, then $\eta\times v=0$.
Since $v\mapsto\eta\times v$ is injective by
Proposition~\ref{prop:products}, this gives $v=0$.

\emph{An arbitrary number of disc factors.}
For $1\le j\le m$, let $\tau_j:\F^{m+1}\to\F^{m+1}$ replace
coordinate $j$ by the identity, and put
$$
 \Pi_m=\prod_{j=1}^m(\id-\tau_j^*).
$$
We follow the calculation in
\cite[Section~5.1, equations~(68)--(70)]{Brandenbursky2026}. 
Expand \eqref{eq:surface-map} using
\eqref{eq:split}--\eqref{eq:local-seed}.
There are exactly $m$ degree-two factors to distribute.
A term using $E_\delta$, or assigning two factors to the
same disc coordinate, leaves another disc coordinate unused
and is sent to zero by $\Pi_m$.
The remaining terms assign one factor to each disc coordinate.
There are $m!$ such assignments, and their signs agree
because the factors have even degree. Thus
$$
\Pi_m\rho_{\eps,\delta}^*\Phi_{m,q}(v)=m!\,\eta_{1,\eps}\times\cdots\times\eta_{m,\eps}\times v_\delta.
$$
Passing to the limit in reduced norm as
$(\eps,\delta)\to(0,0)$ gives
\begin{equation}\label{eq:surface-limit}
\Pi_m\rho_{\eps,\delta}^*\Phi_{m,q}(v)\longrightarrow m!\kappa\Bigl(\prod_{j=1}^m\lambda_j\Bigr)\eta^{\times m}\times v.
\end{equation}

If $\Phi_{m,q}(v)=0$, the left-hand side of
\eqref{eq:surface-limit} is zero for every $\eps,\delta$.
Thus $\eta^{\times m}\times v=0$ in reduced bounded cohomology.
Proposition~\ref{prop:products} gives $v=0$.
For $m=0$, injectivity is \eqref{eq:one-point-limit}.

\emph{Exactness.}
By \eqref{eq:surface-ordinary} and the one-point comparison
identity, $\Theta_q(v)$ lies in $\rEH^q(\cN)$.
Cup products with exact classes are exact before reduction,
since comparison respects cup products. Passing to the reduced
groups shows that $\Phi_{m,q}(v)\in\rEH^{2m+q}(\cN)$.
Thus we have injective maps
$$
\Phi_{m,q}:\rHb^q(\F)\longrightarrow\rEH^{2m+q}(\cN),\qquad m\ge0,\quad q=2,3.
$$
Every integer $d\ge2$ has the form $d=2m+q$.
Equation~\eqref{eq:free-dimensions} proves
Theorem~\ref{thm:main} in dimension two for the measure
$\mu$ induced by the smooth positive density fixed above.
Lemma~\ref{lem:good-measures} gives the assertion for every
good measure.

\section{Proof of Theorem 1 in dimension greater than two}
\label{sec:higher}

\emph{The free subgroup.}
Let $\dim M=n\ge3$. Retain $G=\pi_M=\pi_1(M)$ and
$\cG=\Homeo_0(M,\mu)$.
\begin{lemma}\label{lem:oriented-neighborhood}
There is a hyperbolic embedding $i:\F\hookrightarrow G$
induced by an embedded figure-eight in $M$ with an orientable
regular neighborhood. For every integer $k\ge1$, the same embedding
can be realized, after choosing basepoint paths, by $k$ disjoint
embedded figure-eights with disjoint orientable regular
neighborhoods.
\end{lemma}
\begin{proof}
Let $w_M:G\to\Z/2$ be the orientation character: it is zero
on loops preserving local orientation and one on loops
reversing it.
The group $G$ is torsion-free and nonelementary hyperbolic, and
$\ker w_M$ has index two. By
\cite[Theorem~7.7]{Kapovich1999}, $\ker w_M$ contains a free
subgroup of rank two which is malnormal and quasiconvex in
the full group $G$. Such a subgroup is hyperbolically embedded;
see \cite[Theorem~7.11]{Bowditch2012} and
\cite[Proposition~4.28(a)]{DGO2017}.
Denote the resulting marked inclusion by
$i:\F=\langle a,b\rangle\hookrightarrow G$.

Represent $i(a),i(b)$ by a map of a figure-eight.
General position makes this map an embedding
without changing the induced inclusion on fundamental groups.
Take a regular neighborhood $U$. Since
$i(\F)\subset\ker w_M$, every loop in $U$ preserves local
orientation in $M$. Thus $U$ is orientable.

Apply general position simultaneously to $k$ copies of the
marked graph map inside $U$. The resulting figure-eights are pairwise
disjoint, and have disjoint small regular
neighborhoods.
\end{proof}

\emph{The local maps and Moser's theorem.}
Fix $k\ge1$ and choose an orientation on each neighborhood in
Lemma~\ref{lem:oriented-neighborhood}. There the density defining
$\mu$ is represented by a positive volume form $\omega$.
Write $D_r^{n-1}$ for the closed Euclidean ball of radius $r$,
and put $D^{n-1}=D_1^{n-1}$.
In one neighborhood, the two loops have trivial normal
bundles. Choose tubes $T_a,T_b\cong S^1\times D^{n-1}$
around them whose intersection is a ball.
On the model tube, use coordinates $\vartheta\in\R/\Z$ and
$y=(y_1,\ldots,y_{n-1})\in D^{n-1}$, and put
$$
Q_n=(\R/\Z)\times D^{n-1},\qquad\omega_0=\mathrm d\vartheta\wedge\mathrm dy_1
\wedge\cdots\wedge\mathrm dy_{n-1}.
$$
For $s=a,b$, put $c_s=\mu(T_s)/\int_{Q_n}\omega_0>0$.
Moser's theorem in the boundary version of Banyaga
\cite{Moser1965,Banyaga1974}, applied to these forms of equal
total volume, gives a parametrization
\begin{equation}\label{eq:moser}
\psi_s:T_s\longrightarrow Q_n,\qquad
\psi_s^*(c_s\omega_0)=\omega|_{T_s}.
\end{equation}
Choose the angular directions so that the trajectory cocycle
of one full circuit is $i(a)$, respectively $i(b)$, with the
convention of Section~\ref{sec:preliminaries}. Thus the
geometric circuits represent $i(a)^{-1},i(b)^{-1}$.
For $0<\eps<1$, let $\chi_\eps:D^{n-1}\to[0,1]$ be a smooth function equal to one
on $D_{1-\eps}^{n-1}$ and zero near the boundary. The isotopy
\begin{equation}\label{eq:pushing}
P_\eps^t(\vartheta,y)=(\vartheta+t\chi_\eps(y),y),\qquad 0\le t\le1
\end{equation}
preserves $\omega_0$. Conjugate by $\psi_s$ and extend by
the identity outside $T_s$. The time-one map fixes the inner
region $\psi_s^{-1}((\R/\Z)\times D_{1-\eps}^{n-1})$
pointwise, while its trajectories there make one full circuit.
The complementary collar in $T_s$ has measure tending to zero
as $\eps\to0$.

This is the construction of Brandenbursky--Menashe
\cite[Section~4.5 and Lemma~4.12]{BrandenburskyMenashe2026}
in dimension $n$; see also the tube formula in
\cite[Lemma~3.1]{BrandenburskyMarcinkowski2022}.
Here is where we use the orientability of the figure-eight
neighborhood: it gives the positive volume forms to which
Moser's theorem applies in \eqref{eq:moser}.
Apply this construction in $k$ disjoint neighborhoods
$U_1,\ldots,U_k$. Write $T_{j,a},T_{j,b}\subset U_j$ for the
two tubes and $\psi_{j,s}$ for their parametrizations, where
$s=a,b$. Define
$$
 f_{j,s,\eps}(x)=
 \begin{cases}
 \psi_{j,s}^{-1}P_\eps^1\psi_{j,s}(x),&x\in T_{j,s},\\
 x,&x\notin T_{j,s}.
 \end{cases}
$$
Let $a_j,b_j$ be the standard generators of the $j$th factor
of $\F^k$. We define
\begin{equation}\label{eq:higher-testing-map}
 \rho_\eps:\F^k\longrightarrow\cG,\qquad
 \rho_\eps(a_j)=f_{j,a,\eps},\quad
 \rho_\eps(b_j)=f_{j,b,\eps}.
\end{equation}
For $w_1,\ldots,w_k\in\F$, define
$$
 \rho_\eps(w_1,\ldots,w_k)
 =\prod_{j=1}^k w_j(f_{j,a,\eps},f_{j,b,\eps}),
$$
where each word is evaluated at the two indicated maps.
Since maps corresponding to distinct indices $j$ have
disjoint supports and therefore commute, $\rho_\eps$
is a homomorphism. For each tube define its inner region
$$
A_{j,s,\eps}=\psi_{j,s}^{-1}\bigl((\R/\Z)\times D_{1-\eps}^{n-1}\bigr),\qquad s=a,b.
$$
Let $B_\eps$ be the union of all collars, and let $U_{j,\eps}$
be the overlap of the two inner regions in $U_j$:
$$
B_\eps=\bigcup_{j=1}^k\bigcup_{s=a,b}(T_{j,s}\setminus A_{j,s,\eps}),\qquad
U_{j,\eps}=A_{j,a,\eps}\cap A_{j,b,\eps}.
$$
Then $\mu(B_\eps)\to0$, and every element of
$\rho_\eps(\F^k)$ fixes $M\setminus B_\eps$ pointwise.
In particular, $B_\eps$ is invariant under $\rho_\eps(\F^k)$.
Write $p_j:\F^k\to\F$ for the coordinate projection.
For $x\in U_{j,\eps}$, the trajectory homomorphism
$w\mapsto\gamma(\rho_\eps(w),x)$ is conjugate to $i\circ p_j$.
On the remaining fixed regions its image is cyclic or trivial;
see \cite[Lemma~11]{Brandenbursky2026}.
Moreover,
$$
\lambda_{j,\eps}=\mu(U_{j,\eps})\longrightarrow\lambda_j=\mu(T_{j,a}\cap T_{j,b})>0.
$$
\emph{Transfer and injectivity.}
\enlargethispage{3\baselineskip}
For the inclusion $i$ above, the extension theorem of
Frigerio--Pozzetti--Sisto
\cite[Corollary~2 and Proposition~5.1]{FrigerioPozzettiSisto2015}
gives, for $q=2,3$, algebraic linear sections
$$
J_q:\rHb^q(\F)\longrightarrow\rEH^q(G),\qquad i^*J_q=\id;
$$
see \cite[Lemma~6]{Brandenbursky2026}.
Choose the nonzero class $\theta\in\rHb^2(\F)$ supplied by
\cite[Lemma~7]{Brandenbursky2026}. It satisfies
$$
\theta\smile u=0\qquad(u\in\rHb^s(\F),\ s\ge1).
$$
In this section, $T_m^q$ denotes the map of
Proposition~\ref{prop:products} for this choice of $\theta$.
For $m\ge0$ and $q=2,3$, put $d=2m+q$, $k=m+1$, and define
\begin{equation}\label{eq:higher-map}
 \Phi_{m,q}:\rHb^q(\F)\longrightarrow\rEH^d(\cG),\qquad
 \Phi_{m,q}(v)=\overline\Gamma_{b,k}^{\,d}\bigl(J_2(\theta)^{\times m}\times J_q(v)\bigr).
\end{equation}
The source factors have representatives in exact bounded
cohomology. External products commute with comparison, and
the product transfer preserves exactness.
Hence $\Phi_{m,q}(v)\in\rEH^d(\cG)$.

Permutation invariance of the product measure
\cite[equation~(26)]{Brandenbursky2026} gives
$$
\Phi_{m,q}(v)=\frac1{k!}\,\overline\Gamma_{b,k}^{\,d}
\bigl(\Sym_k(J_2(\theta)^{\times m}\times J_q(v))\bigr).
$$
Put $C_k=\prod_{j=1}^k\lambda_j>0$ and apply
\cite[Proposition~12]{Brandenbursky2026} to $\rho_\eps$.
The proof partitions $M^k$ according to the trajectory types
just described. The set where an integration coordinate lies
in $B_\eps$ has measure at most $k\mu(B_\eps)$, so its
contribution tends to zero in norm. On its complement, 
the point-dependent conjugations do not
affect the resulting bounded cohomology class; see
\cite[proof of Proposition~12]{Brandenbursky2026}.
A term with a cyclic or trivial trajectory factor vanishes
in positive degrees. If two integration coordinates use the
same free factor, restriction contains either
$\theta\smile\theta$ or $\theta\smile v$, and hence vanishes.
The surviving regions are the $k!$ products obtained by
permuting $U_{1,\eps},\ldots,U_{k,\eps}$, each with measure
$\prod_j\lambda_{j,\eps}$. Their sum gives $T_m^q(v)$, with
the normalization above. Consequently,
\begin{equation}\label{eq:higher-limit}
\rho_\eps^*\Phi_{m,q}(v)\longrightarrow C_k T_m^q(v)\qquad(\eps\to0)
\end{equation}
in reduced norm for every fixed $v\in\rHb^q(\F)$.

If $\Phi_{m,q}(v)=0$, then \eqref{eq:higher-limit} and $C_k>0$
give $T_m^q(v)=0$. Proposition~\ref{prop:products} implies
$v=0$, so $\Phi_{m,q}$ is injective. By \eqref{eq:free-dimensions},
$$
\dim\rEH^{2m+q}(\cG)=\infty\qquad(m\ge0,\ q=2,3).
$$
Every integer $d\ge2$ has this form.
Lemma~\ref{lem:good-measures} gives the assertion for the
original good measure. Finally, $\rEH^d(\cG)$ is a quotient
of $EH_b^d(\cG)$, so the unreduced assertion follows.
This completes the proof of Theorem~\ref{thm:main}.

\bibliographystyle{plain}
\bibliography{references}
\end{document}